\documentclass[10pt,a4paper,twoside]{amsart}
\usepackage[T1]{fontenc}
\usepackage[utf8]{inputenc}
\usepackage[english]{babel}
\usepackage{amsmath}
\usepackage{amsthm}
\usepackage{amssymb}
\usepackage{amsfonts}
\usepackage{amsxtra}
\usepackage{enumerate}
\usepackage{verbatim}
\usepackage{color}
\newcommand{\Z}{\mathbb Z}
\newcommand{\Q}{\mathbb Q}
\newcommand{\R}{\mathbb R}
\newcommand{\C}{\mathbb C}
\newcommand{\I}{\mathcal{I}}
\newcommand{\J}{\mathcal{J}}
\newcommand{\K}{\mathcal{K}}
\newcommand{\RR}{\mathcal{R}}

\newtheorem*{conj}{Conjecture}
\newtheorem{theorem}{Theorem}
\newtheorem{definition}[theorem]{Definition}
\newtheorem{lemma}[theorem]{Lemma}

\newtheorem{prop}[theorem]{Proposition}

\theoremstyle{remark}
\newtheorem{remark}[theorem]{Remark}

\DeclareMathOperator*{\Light}{Light}

\let\phi=\varphi
\let\e=\varepsilon

\begin{document}

\setcounter{page}{1}

\title{A Counterexample to Guillemin's Zollfrei Conjecture}

\author[Stefan Suhr]{Stefan Suhr}
\address{Fachbereich Mathematik, Universit\"at Hamburg}
\email{stefan.suhr@math.uni-hamburg.de}

\date{\today}

\begin{abstract}
We construct Zollfrei Lorentzian metrics on every nontrivial orientable circle bundle over a orientable closed surface. Further we prove a weaker version of 
Guillemin's conjecture assuming global hyperbolicity of the universal cover. 
\end{abstract}
\maketitle

\section{Introduction}\label{last}

In\footnote{53C22, 53C50, closed geodesics, Lorentzian manifolds} \cite{guill1} a pseudo-Riemannian metric $g$ on a compact manifold $M$ is called {\it Zollfrei} if the geodesic flow of $(M,g)$ induces the structure of a fibration 
by circles on the set of lightlike nonzero vectors $\{g(v,v)=0,v\neq 0\}\subseteq TM$. In particular every lightlike geodesic of a Zollfrei metric is closed. We will
call manifolds that admit a Zollfrei metric {\it Zollfrei} as well. 

It is not difficult to classify all Zollfrei surfaces. \cite{guill1} shows that every Zollfrei surface $(M,g)$ admits a finite cover $(M',g')$, which is globally conformal to 
$(\R^2/\Z^2, \underline{dxdy})$ where $x$ and $y$ are the canonical coordinates on $\R^2$ and $\underline{dxdy}$ denotes the metric induced by $dxdy:=
\frac{1}{2}(dx\otimes dy+dy \otimes dx)$. However already for $3$-manifolds the questions of determining the diffeomorphism type of a manifold admitting a 
Zollfrei metric is wide open. Besides the examples described in \cite{guill1}, called {\it standard examples}, none were 
known so  far. The manifolds of the standard examples all have one of the four diffeomorphism types of compact manifolds with universal cover $S^2\times\R$ 
(see \cite{tollefson}), i.e. $S^2\times S^1$ and the three manifolds with double cover $S^2\times S^1$ corresponding to the three involutions of $S^2\times S^1$:
\begin{itemize}
\item[(i)] $(x,y,z,t)\mapsto (-x,-y,-z,t)$
\item[(ii)] $(x,y,z,t)\mapsto (-x,-y,-z,t+\pi)$ 
\item[(iii)] $(x,y,z,t)\mapsto (-x,-y,z,t+\pi)$ .
\end{itemize}
Here the $2$-sphere is considered as the submanifold $\{x^2+y^2+z^2=1\}$ of $\R^3$. The standard examples all lift to a metric $g_{can}-\lambda dt^2$ on 
$S^2\times \R$, where $g_{can}$ is the canonical round sphere metric of radius one and $\lambda >0$. 

In 1989 V. Guillemin conjectured in his book \cite{guill1}, page 8:
\begin{conj}[\cite{guill1}]
Every Zollfrei manifold in dimension three has the same diffeomorphism type as one of the standard examples.
\end{conj}

In this article we will give a counterexample to the Conjecture.
\begin{theorem}\label{T1}
Every nontrivial orientable circle bundle over a closed and orientable surface admits Zollfrei metrics.
\end{theorem}

Note that the Gysin sequence for the integral cohomology implies that the different circle-bundles over different surfaces are nonhomeomorphic.

It is well known that the lightlike geodesics of a pseudo-Riemannian manifold are invariant, as unparameterized curves, under global conformal changes of the 
metric. Therefore Theorem \ref{T1} in fact yields an infinite-dimensional family of Zollfrei metrics.

\section{Proof}\label{last4}

\subsection{The Theorem of Boothby and Wang \cite{boowa}}
Following \cite{geiges}, section 7.2., we call a vector field {\it regular} if around every point there exists a flow box that every flow line 
intersects at most once. We call a contact form $\alpha$ on an odd dimensional manifold {\it regular} if the Reeb vector field $\RR$ of $\alpha$ 
is regular. 

Let $\pi\colon M\to B$ be a principal circle bundle and $\RR\in\Gamma(TM)$ tangent to the fibres with $2\pi$-periodic flow, i.e. $\RR$ generates the  
circle action on $M$. Then, following \cite{geiges}, a differential $1$-form $\alpha\in \Omega^1(M)$ is a {\it connection $1$-form} if 
(1) $\mathcal{L}_\RR (\alpha)\equiv 0$ and
(2) $\alpha(\RR)=1$. Note that here we explicitly assume the identification $i\cdot\R\cong \R$. Conditions (1) and (2) imply that there exists a well defined closed
$2$-form $\omega$, called the {\it curvature form}, defined by $\pi^*\omega =d\alpha$. The class $-\left[\frac{\omega}{2\pi}\right]\in H^2(B,\Z)$ 
is called the {\it Euler class} of the bundle $\pi\colon M\to B$.

\begin{theorem}[\cite{geiges}, Theorem 7.2.4]
Let $(B,\omega)$ be a closed symplectic manifold with integral symplectic form $\omega/2\pi$. Let $\pi\colon M \rightarrow B$ be the 
principal $S^1$-bundle with Euler class $-\left[\frac{\omega}{2\pi}\right] \in H^2(B,\Z)$. Then there is a connection $1$-form $\alpha$ on $M$ with the following 
properties:
\begin{itemize}
\item $\alpha$ is a regular contact form,
\item the curvature form of $\alpha$ is $\omega$,
\item the vector field $\RR$ defining the principal circle action on $M$ coincides with the Reeb vector field of $\alpha$.
\end{itemize}
\end{theorem}

Let $(B,g)$ be a smooth closed surface with constant curvature $K$ and volume form $\text{dvol}^g$ such that $\text{vol}(B,g)\in 2\pi \Z$.
Then $\frac{\text{dvol}^g}{2\pi}$ is an integral symplectic form. Denote with $\pi\colon M\to B$ the $S^1$-principle bundle over $B$ with Euler class 
$-\left[\frac{\text{dvol}^g}{2\pi}\right]\in H^2(B,\Z)$. Define for $\phi \in (0,\frac{\pi}{2})$ the Lorentzian metric 
$$h_\phi :=\pi^* g-\cot^2\phi\cdot \alpha\otimes\alpha$$
on $M$. Note that $\pi-2\phi$ is the opening angle of the light cones of $h_\phi$ around $\RR$. 
The Reeb vector field $\RR$ of $\alpha$ is a timelike Killing vector field of $h_\phi$, i.e. $(M,h_\phi)$ is stationary spacetime. 
Denote with $\Phi$ the flow of $\RR$. 

Note that in the case of $B\cong S^2$ and $K=4$ we can describe $h_\phi$ as the pseudo-Riemannian analogues to the Berger spheres \cite{berger}. 
Consider the canonical embedding $i\colon S^3\hookrightarrow \mathbb{H}\cong \R^4$ into the Quaternions and the orthonormal frame field $(\mathcal{I},
\mathcal{J},\mathcal{K})$ with $\I\colon x\mapsto i\cdot x$, $\J\colon x\mapsto j\cdot x$ and $\K\colon x\mapsto k\cdot x$. By $(\I^\ast,\J^\ast,\K^\ast)$ denote the 
dual frame field. Let $\langle.,.\rangle$ be the canonical scalar product on $\R^4$. Then we know that $\I^*=\alpha$ and $\RR=\I$ for $(B,g)=(\C P^1,g_{FS})$ 
where $g_{FS}$ denotes the Fubini-Studi metric. Further it follows that  
$$h_\phi(x):=i^*\langle .,.\rangle -\frac{1}{\sin^2\phi}\I^\ast \otimes \I^\ast=\J^\ast\otimes \J^\ast+\K^\ast\otimes \K^\ast-\cot^2(\phi) \I^\ast\otimes \I^\ast.$$

\subsection{The Arrival Time}\label{S2.2}

Following \cite{cjm} we will call a Lorentzian manifold $(\mathcal{M},g)$ {\it standard stationary} iff $\mathcal{M}$ splits into a product $\mathcal{M}_0\times \R$ 
with
\begin{equation}\label{NE1}
g(x,t)[(v,\tau),(v,\tau)]=g_0(x)[v,v]+2g_0(x)[\delta(x),v]\tau-\beta(x)\tau^2
\end{equation}
where $(x,t)\in \mathcal{M}_0\times \R$, $(v,\tau)\in T_x \mathcal{M}_0\times \R$, $g_0$ is a Riemannian metric on $\mathcal{M}_0$, $\delta$ a smooth vector 
field on $\mathcal{M}_0$ and $\beta$ a positive smooth function on $\mathcal{M}_0$.

Denote with $F$ the Finsler metric on $\mathcal{M}_0$ given by 
$$F(x,v):=\sqrt{\tilde{g}_0(x)[v,v]+\tilde{g}_0(x)[\delta(x),v]^2}+\tilde{g}_0(x)[\delta(x),v]$$
where $\tilde{g}_0:=g_0/\beta$. We have the following {\it Fermat's principle}:

\begin{theorem}[\cite{cjm} Theorem 4.1]\label{T2}
Let $(\mathcal{M},g)$ be a standard stationary spacetime and $(x_0,t_0)\in \mathcal{M}$, $s\in\R\mapsto \gamma(s)=(x_1,s)\in \mathcal{M}$, 
$x_1\in \mathcal{M}_0$.
A curve $s\in [0,1]\to z(s)=(x(s),t(s))\in\mathcal{M}$ is a future pointing lightlike geodesic of $(\mathcal{M},g/\beta)$ if and only if $x(s)$ is a geodesic for the 
Fermat metric $F$, parameterized to have constant Riemannian speed $h(x)[\dot{x},\dot{x}]=\widetilde{g}_0(x)[\delta(x),\dot{x}]^2+\widetilde{g}_0[\dot{x},\dot{x}]$,
and 
\begin{equation}\label{NE2}
t(s)=\int_0^s \left(\widetilde{g}_0(x)[\delta(x),\dot{x}]+\sqrt{\widetilde{g}_0(x)[\delta(x),\dot{x}]^2+\widetilde{g}_0[\dot{x},\dot{x}]}\right)d\nu.
\end{equation}
\end{theorem}
Recall that the lightlike geodesics of $(\mathcal{M}_0\times \R,g)$ are reparameterizations of the lightlike geodesics of $(\mathcal{M}_0\times \R,\tilde{g})$.

\subsection{The Global Construction}
We want to apply the reasoning of Theorem \ref{T2} to the lightlike geodesics of $(M,h_\phi)$. It is clear that these Lorentzian manifold are not standard 
stationary. But we can consider the problem locally. Choose a $h_\phi$-spacelike immersion $I\colon D^2\to M$, i.e. $I^* h_\phi >0$. Then the map 
$\Phi_I\colon D^2\times \R\to M$, $(p,t)\mapsto \Phi(I(p),t)$ is an immerison of $D^2\times \R$. It follows that $(D^2\times \R,(\Phi_I)^\ast h_\phi)$ is a 
standard stationary Lorentzian spacetime in the sense of \cite{cjm}. 
\begin{remark}\label{R2}
Consider the arrival time $t_x$ for a closed curve $x\colon[a,b]\to D^2$ as defined by \eqref{NE2}. Then $[t_x(b)-t_x(a)] \text{mod}\;2\pi$
quantifies the obstruction of the lightlike curve $s\mapsto \Phi_I(x(s),t_x(s))$ to being closed in $M$. This follows from the fact that the map $\Phi_I$ is 
$2\pi$-periodic in the $\R$-factor. So if the arrival time is a rational multiple of $2\pi$, e.g. $t_x(b)-t_x(a)=\frac{p}{q}2\pi$, then 
$s\mapsto \Phi(x^q(s),t_{x^q}(s))$ is closed.
\end{remark}

We will now give a {\it coordinate invariant} description of the arrival time, i.e. a functional for curves in the orbit space $B$. 
Let $I\colon D^2\to M$ be a spacelike immersion as before. We express the components of \eqref{NE1} in terms of objects defined on $B$. 
We immediately see that $\widetilde{g}_0= \frac{1}{-h_\phi(\RR,\RR)}I^* h_\phi$. Next define $\underline{Y}:=Y-\frac{h_\phi (Y,\RR)}{h_\phi (\RR,\RR)} \RR$ for 
$Y\in TM$. Then on one side we have $h_\phi(\underline{Y},\underline{Y})=g(\pi_\ast Y,\pi_\ast Y)$. One the other side we have 
$$h_\phi(\underline{Y},\underline{Y})=-h_\phi(\RR,\RR)\left(\frac{h_\phi(Y,Y)}{-h_\phi(\RR,\RR)}+\frac{h_\phi(Y,\RR)^2}{h_\phi(\RR,\RR)^2}\right).$$
From \eqref{NE1} follows that $\widetilde{g}_0(\delta(x),.)=\frac{I^* h_\phi(\RR,.)}{-h_\phi(\RR,\RR)}=-I^*\alpha$. Combining the equations we get
$$F(x,v)=\sqrt{\frac{(\pi\circ I)^*g(v,v)}{-h_\phi(\RR,\RR)}}-I^*\alpha(v)=\tan\phi\cdot\sqrt{(\pi\circ I)^*g(v,v)}-I^*\alpha(v).$$

Spacelike immersions $I\colon D^2\to M$ are induced for example by certain local section of $\pi\colon M\to B$. Therefore for $U\subseteq B$ contractible and 
a section $s\colon U\to M$ such that the image of $s$ is spacelike, we can define $F_U$ via the previous formula. Note that a global version of $F_U$ does not 
exists, as $M$ is nontrivial. In our case this comes from the fact that $\alpha$ has no well defined counterpart on $B$. But $d\alpha$ has a well defined 
counterpart on $B$ in the form of $\text{dvol}^g$. So the global version of the arrival time functional, we are looking for, is no longer well defined on curves 
$\gamma\colon I\to B$, but instead is defined for (e.g.) smooth maps $f\colon S\to B$ where $S$ is a compact oriented $2$-manifold with nonempty boundary. 
From this point of view the {\it arrival time functional} for a map $f\colon S\to B$ is
$$cp_\phi(f):=\tan\phi\cdot L^{g}(f|_{\partial S})+\int_S f^\ast (-\text{dvol}^{g})$$
where $L^{g}$ denotes the length functional of $g$. Following \cite{tai1} we will call functionals like $cp_\phi$ {\it charged particles}.

The critical points of $cp_\phi$ describe the periodic orbits of a charged particle on $(B,\tan(\phi)\cdot g)$ moving under the influence of the 
``magnetic field'' $-\text{dvol}^{g}$. 

\subsection{Extremals of $cp_\phi$}
The general framework can be described as follows: The definitions and results are taken from \cite{tai0}. Let $\omega$ be an arbitrary $2$-form on $B$. 
Consider for the class of contractible open subsets $U\subseteq B$ the family of functionals 
$$cp_U\colon C^{\infty}(I,U)\to \R,\; \gamma\mapsto cp_U(\gamma):=L^{g}(\gamma)+\int_\gamma \sigma,$$
where $\sigma \in \Omega^1(U)$ is a primitive of $\omega|_{U}$. The physical intuition behind $cp_U$ is that the critical points of $cp_U$ model the motion 
of a charged particle moving under the influence of a magnetic field $\omega$. In physics terms the $1$-form~$\sigma$ represents a vector potential of 
the magnetic field $\omega|_U$. 

\begin{remark}
Let $\gamma\colon [a,b]\to U$ be a regular curve parameterized w.r.t. constant $g$-arclength. 
Then
\begin{equation}\label{NE3}
\frac{1}{|\dot{\gamma}|_g}\nabla_{\dot{\gamma}} \dot{\gamma}=-\omega(\dot{\gamma},.)^{\#} 
\end{equation}
are the Euler-Lagrange equations of $cp_U$ (Here we set $|v|_g:=\sqrt{g(v,v)}$). Especially the Euler-Lagrange equations are independent of $U$ and $\sigma$. 
We will refer to the solutions as {\it extremals}. 
\end{remark}

Recall that there exist constant $\e,\delta >0$ such that any pair of points $x,y\in B$, with distance at most $\delta$, can be joined by a unique solution $\gamma
\colon [0,1]\to B$ of \eqref{NE3} lying completely in the ball of radius $\e$ around $x$, compare \cite{tai0}.

Next we describe the relation between the critical points of $cp_\phi$ and $cp_U$. The following definition is taken from \cite{tai0}. A {\it film} $\Pi\subseteq B$ is 
an oriented surface with nonempty boundary embedded into $B$ such that the boundary is a union of finitely many closed curves 
$\gamma_\alpha$ with the following properties: 

\begin{itemize}
\item[1)] $\gamma_\alpha$ has the boundary orientation,
\item[2)] every $\gamma_\alpha$ is a finite polygon of extremal segments whose lengths do not exceed $\delta$ and 
\item[3)] the curves $\gamma_\alpha$ are disjoint, i.e. $\gamma_\alpha \cap \gamma_\beta=\emptyset$ if $\alpha\neq \beta$.
\end{itemize}
The space of films on $B$ is denoted with $L(B)$. Define 
$$cp\,\colon\, L(B)\to \R,\; \Pi\mapsto L^g(\Pi|_{\partial S})+\int_S \Pi^\ast \omega.$$

\begin{lemma}[\cite{tai0}, Lemma 1]
Let $\Pi\in L(B)$. Then $\delta cp(\Pi)=0$, i.e. $\Pi$ is a critical point of $cp$, if and only if the boundary of $\Pi$ consists of a union of smooth closed 
extremals.
\end{lemma}

As an example we calculate the extremals of $(\C P^1, \tan\phi\cdot g_{FS},\text{dvol}^{FS})$. The Fubini-Study metric is a Riemannian metric on $\C P^1\cong S^2$ of 
constant curvature equal to $4$ and diameter $\pi/2$. Hence $(\C P^1, g^{FS})$ is isometric to $(S^2,\frac{1}{4}g_{can})$ where 
$g_{can}$ denotes the canonical metric on $S^2$ with curvature equal to one. Fix the ``polar'' parameterization 
$$P\colon U:=(0,\pi)\times (0,2\pi)\to S^2, \;(\theta,\psi)\mapsto (\sin\theta \cos\psi, \sin\theta \sin\psi, \cos\theta).$$ 

We claim that one solution to the Euler-Lagrange equations of $cp_U$ is the curve $t\mapsto (\sin\theta \cos t, \sin\theta \sin t, \cos\theta)$ with $\theta
=\arctan(2\tan\phi)$. Note that this curve is simply closed. Observe that $P^\ast g^{FS}=\frac{1}{4}(d\theta^2+\sin^2\theta d\psi^2)$ and $P^\ast\text{dvol}^{FS}=
\frac{\sin\theta}{4}d\theta\wedge d\psi$. Further observe that $\Gamma^\psi_{\psi \psi}=0$ and $\Gamma^\theta_{\psi \psi}=-\sin(\theta)\cos(\theta)$. Consequently 
we have
$$
\frac{\tan\phi}{|\dot{\gamma}|}\nabla_{\partial_\psi}\partial_\psi
=-2\cos(\theta) \tan(\phi) \partial_\theta$$
and 
$$-\omega(\dot{\gamma},.)^\# =\text{dvol}^{FS}(\partial_\psi,.)^\#=-\sin(\theta) \partial_\theta.$$

Then the curve $t\mapsto (\sin\theta \cos t, \sin\theta \sin t, \cos\theta)$ is a solution of the Euler Lagrange equation iff $-2\tan\phi \cos\theta=
-\sin\theta$ or equivalently $\theta =\arctan (2\tan\phi)$. 


\begin{remark}\label{R3}
Recall that if $(B,g)$ has constant curvature, the isometry group of the universal cover $(\widetilde{B},\widetilde{g})$ acts transitively on $T^1\widetilde{B}$. 
Further it preserves $\widetilde{cp}$, the lift of $cp$, if $\omega$ is a multiple of $\text{dvol}^g$. This implies that the extremals of $\widetilde{cp}$ are either all 
closed or all nonclosed. In the case of nonnegative curvature all extremals of $\widetilde{cp}$, and therefore $cp$, are closed. For curvature equal to $-1$ the 
extremals are closed iff $|\omega(\dot{\gamma},.)^\sharp|_g>1$. Note that it is sufficient to consider the case $K=-1$. 
\end{remark}

\subsection{Lightlike Geodesics of $(M,h_\phi)$}

Since $cp_\phi$ is constant on all simply closed extremals, Remark \ref{R2} readily implies:

\begin{prop}\label{P2.1}
The lightlike geodesics of $(M,h_\phi)$ are either all closed or all nonclosed depending on whether the extremals of $cp_\phi$ are closed
and $cp_\phi(\gamma)\in \Q$ for any extremal $\gamma$ of $cp_\phi$. 
\end{prop}

Again as an example we determine the condition for $(\C P^1,\tan\phi \cdot g_{FS},\text{dvol}^{FS})$.
In the polar parameterization $P$ we can choose the primitive $\sigma=-\frac{\cos\theta}{4}d\psi$ of $P^\ast(-$dvol$^{FS})$. Thus we have 
$\int_\gamma \sigma =-\frac{\pi}{2} \cos(\arctan (2\tan\phi))$. The length of every simply closed solution of \eqref{NE3} is 
$\pi \sin(\arctan (2\tan\phi))$. Bringing everything together we get
\begin{align*}
cp_U(\gamma)&=\pi\left(\sin(\arctan (2\tan\phi))-\frac{\cos(\arctan (2\tan\phi))}{2}\right)\\
&=\frac{\pi}{2} \cos(\arctan (2\tan\phi))(4\tan(\phi)-1)\\
&=\frac{\pi}{2}\frac{1}{\sqrt{1+4\tan^2\phi}}(4\tan(\phi)-1)
\end{align*}
From Remark \ref{R2} we know that $cp_U(\gamma)$ mod $2\pi$ quantifies the obstruction for the lightlike geodesics over $\gamma$ to being closed. Therefore 
the lightlike geodesics of $(S^3,h_\phi)$ will eventually close iff $\frac{4\tan(\phi)-1}{\sqrt{1+4\tan^2\phi}}\in \Q$.

\begin{prop}\label{P2.2}
If the lightlike geodesics of $(M,h_\phi)$ are closed then $(M,h_\phi)$ is Zollfrei.
\end{prop}

\begin{proof}
First recall that \eqref{NE3} defines an Euler-Lagrange flow on $TB$. The point of the proof is to note the relation between the Euler-Lagrange flow of $cp_\phi$ 
on $TB^\times:=TB\setminus\{\text{zero section}\}$ and the geodesic flow of $h_\phi$ on the smooth manifold $\Light_\phi$ of lightlike tangent vectors $v\neq 0$ 
of $(M,h_\phi)$. More precisely the Euler-Lagrange flow $\Phi_B$ of $\cot\phi\cdot cp_\phi$ on $TB^\times$ and the geodesic flow 
$\Phi_M$ on $\Light_\phi$ are conjugated via $\pi_*$, i.e. $\pi_*\circ \Phi_M=\Phi_B\circ \pi_*$, where $\pi\colon M\to B$ denotes the bundle projection.

1) $\Phi_B$ induces an circle fibration on $TB^\times$: Consider the pullback $\widetilde{cp}_\phi$ of $cp_\phi$ to the universal cover 
$\widetilde{\pi}\colon \widetilde{B}\to B$. $\widetilde{cp}_\phi$ is invariant under the action of $\text{Isom}(\widetilde{B},\widetilde{g})$, the isometry group of 
$(\widetilde{B},\widetilde{g})$. Recall that since $\widetilde{g}$ has constant curvature, $\text{Isom}(\widetilde{B},\widetilde{g})$ acts transitively on 
$T^r \widetilde{B}$ for every $r>0$ and commutes with $\Phi_{\widetilde{B}}$, the Euler-Lagrange flow of $\widetilde{cp}_\phi$. The isotropy group of every 
flowline is closed. This defines a fibre bundle structure on every $T^r \widetilde{B}$ with $1$-dimensional fibre. Since 
$\Phi_{\widetilde{B}}(v,t)=\Phi_{\widetilde{B}}(\lambda\cdot v,\frac{t}{\lambda})$, these fibrations extend to a fibration of $T\widetilde{B}^\times$. 
The assumption that the lightlike geodesics of $h_\phi$ are closed implies that the extremals of $cp_\phi$ are closed as well. From Remark \ref{R3} we know that 
in this case the extremals of $\widetilde{cp}_\phi$, and with it the flowlines of $\Phi_{\widetilde{B}}$, are closed as well. Therefore the isotropy groups are 
compact, i.e. the fibres are diffeomorphic to $S^1$. 

The fibration structure is of course invariant under the induced action of $\pi_1(B)$. Therefore it descends to a fibration of $TB^\times$ over the smooth manifold
of flowlines of $\Phi_B$.

2) $\Phi_M$ induces an circle fibration on $\Light_\phi$: Since $\pi_*$ conjugates $\Phi_M$ with $\Phi_B$, every flowline of $\Phi_M$ induces a finite covering 
of the respective flowline of $\Phi_B$. We have seen in Remark \ref{R3} that $cp_\phi$ is constant on every extremal of minimal period. Then Remark \ref{R2} 
implies that these coverings all have the same number of leaves. Together with part 1) this yields that $\Phi_M$ induces the structure of an circle fibration on 
$\Light_\phi$, i.e. $(M,h_\phi)$ is Zollfrei.
\end{proof}

\section{A weaker conjecture}

Despite Theorem \ref{T1} one can hope to prove a weaker version of the conjecture, e.g. assuming additional properties of the pseudo-Riemannian universal 
cover. Note that for $3$-manifolds, up to sign, every pseudo-Riemannian metric, that is not Riemannian or anti-Riemannian, is Lorentzian. This opens for us the
possibility to use notions of causality theory from Lorentzian geometry. \cite{guill1} raises the question, whether additionally assuming causality of the universal 
cover is sufficient for the Conjecture to be true. We can give the following partial answer:

\begin{theorem}\label{T1a}
If the $3$-manifold $M$ admits a Zollfrei metric $g$ such that the universal Lorentzian cover is globally hyperbolic, 
$M$ is covered by $S^2\times S^1$, i.e. diffeomorphic to either $S^2\times S^1$, $\R P^2\times S^1$, $\R P^3 \sharp \R P^3$ or the nonorientable $2$-sphere 
bundle over $S^1$. 
\end{theorem}

Recall that according to \cite{besa3} we can define a Lorentzian manifold $(M,g)$ to be {\it globally hyperbolic} iff $(M,g)$ is isometric 
to $(N\times \R, g_0+g_0(\delta,.)-\beta dt^2)$, where $\beta$ is a smooth positive function on $N\times \R$, $g_0$ is a Riemannian metric on $N$ and $\delta$ 
is a smooth vector field on $N$ both i.g. depending on the $t$-coordinate. Note that global hyperbolicity implies causality for a Lorentzian manifold.

Theorem \ref{T1a} follows from a result due to Low (\cite{low2}, Theorem 5). The proof leans on the notion of {\it refocussing spacetimes} introduced in \cite{low1}.
\begin{definition}[\cite{chekrud08}, Definition 22]
A strongly causal spacetime $(M, g)$ (that is not necessarily globally hyperbolic) is called refocussing at $p\in M$ if there exists a neighborhood $O$ of $p$
with the following property: For every open $U$ with $p\in U\subseteq O$ there exists $q \notin U$ such that all the lightlike geodesics through $q$ enter $U$. 
A space-time $(M, g)$ is called refocussing if it is refocussing at some $p$, and it is called nonrefocussing if it is not refocussing at
every $p \in M$.
\end{definition}
Recall that global hyperbolicity implies strong causality for Lorentzian manifolds. So the definition is not needed in full generality, i.e. those who are not
familiar with causality theory might as well substitute global hyperbolicity for strong causality in the definition.

\begin{theorem}[\cite{low2}, Theorem 5]\label{TL1}
Let $M$ be globally hyperbolic, with non-compact Cauchy hypersurface $N$. Then $M$ cannot be refocussing.
\end{theorem}

With this at hand we can state the following proposition for our purposes.

\begin{prop}
Let $(M,g)$ be a Lorentzian manifold such that the universal Lorentzian cover $(\widetilde{M},\widetilde{g})$ is globally hyperbolic.
If there exists $p\in M$ such that all lightlike geodesics eminating from $p$ return to $p$ with uniformly bounded Riemannian arclength (w.r.t. 
a fixed complete Riemannian metric on $M$), then $M$ is compact and $\widetilde{M}$ is spatially compact, i.e. $\widetilde{M}\cong N\times \R$ with $N$ compact.
\end{prop}

\begin{proof}
According to Theorem \ref{TL1} the only points we have to prove are (1) $(\widetilde{M},\widetilde{g})$ is refocussing and (2) $M$ is compact. 

(1). Since the lightlike geodesic loops around $p$ have bounded Riemannian arclength, they share a common fundamental class $\eta\in \pi_1(M)$. Notice that 
$\eta$ is nontrivial, since else the universal cover would violate causality. Hence the universal cover is refocussing at any point $\widetilde{p}\in 
\widetilde{\pi}^{-1}(p)$, since all lightlike geodesics through $\eta^{-1}(\widetilde{p})$ meet $\widetilde{p}$ and there exists a neighborhood of $\widetilde{p}$ that 
does not contain $\eta^{-1}(\widetilde{p})$.

(2). The deck transformation group of the universal cover acts properly discontinuously on the universal cover. This implies that there exists a $k\in \Z$ such that 
$\eta^k(N\times\{0\})$ is disjoint from $N\times \{0\}$, as $N$ is compact. This implies that the quotient of $\widetilde{M}\cong N\times \R$ by the group generated 
by $\eta^k$ is compact and moreover covers $M$. Then $M$ has to be compact. 
\end{proof}

\begin{proof}[Proof of Theorem \ref{T1a}]
Theorem \ref{TL1} implies for $3$-manifolds that any Cauchy hypersurface in the universal cover has to be diffeomorphic to $S^2$. Thus 
$\widetilde{M}\cong S^2\times \R$. The compact quotients of $S^2\times \R$ were classified in \cite{tollefson}. They are exactly $S^2\times S^1$, 
$\R P^3\times S^1$, $\R P^3 \sharp \R P^3$ and the nonorientable $2$-sphere bundle over $S^1$. 
\end{proof}

At the end of these notes we want to post some questions in connection with Zollfrei $3$-manifolds. P. Mounoud asked if every Zollfrei $3$-manifold is a 
Seifert fibration. At this point the author does not have an idea how to prove such a claim or how to give a counterexample. If a counterexample exists, then 
it cannot be stationary, since \cite{fjp} show that every compact stationary Lorentzian manifold is a Seifert fibration. 

On the other hand is the question if every Seifert fibration admits a Zollfrei Lorentzian manifold. This question is especially interesting for the 
trivial bundles over surfaces of genus greater than one. Again such examples cannot be stationary, since then the Finsler metric $F$ (see section \ref{S2.2})
will be globally well 
defined and the lightlike geodesics correspond the geodesics of $F$ on the underlying surface. Since the fundamental groups of the surfaces in question 
are nontrivial, not all lightlike geodesics can be homotopic. This clearly contradicts the Zollfrei property. 

{\bf Acknowledgment:} I would like to thank Kai Zehmisch and Pierre Mounoud for their careful reading of the first draft and their valuable comments.

\bibliographystyle{amsalpha}

\end{document}